\documentclass{amsart}

\usepackage{amsrefs}
\usepackage{amsthm}
\usepackage{amssymb}
\usepackage{amsfonts}
\usepackage{amsmath}
\usepackage{mathrsfs}
\usepackage{comment}
\usepackage{hyperref}

\numberwithin{equation}{section}

\newtheorem{theorem}{Theorem}[section]
\newtheorem{lemma}[theorem]{Lemma}

\theoremstyle{definition}

\newcommand{\C}{\ensuremath{\mathbb{C}^n}}

\newcommand{\N}{\ensuremath{\mathbb{N}}}
\renewcommand{\S}{\ensuremath{\text{Sig}_d}}
\newcommand{\Rd}{\ensuremath{{\mathbb{R}^d}}}
\newcommand{\R}{\ensuremath{{\mathbb{R}}}}
\newcommand{\D}{\ensuremath{\mathscr{D}}}

\newcommand{\J}{\ensuremath{\mathscr{J}}}
\newcommand{\F}{\ensuremath{\mathscr{F}}}
\newcommand{\G}{\ensuremath{\mathscr{G}}}

\newcommand{\Mn}{\ensuremath{\mathcal{M}_{n }}(\mathbb{C})}

\newcommand{\inrd}{\ensuremath{\int_{\Rd}}}
\newcommand{\inr}{\ensuremath{\int_{\R}}}

\newcommand{\MC}[1]{\ensuremath{\mathcal{#1}}}

\begin{document}

\title[Matrix  weighted Triebel-Lizorkin bounds]{Matrix  weighted Triebel-Lizorkin bounds: a simple stopping time proof}


\author[Joshua Isralowitz]{Joshua Isralowitz}
\address[Joshua Isralowitz]{Department of Mathematics and Statistics \\
SUNY Albany \\
1400 Washington Ave. \\
 Albany, NY  \\
12222}
\email[Joshua Isralowitz]{jisralowitz@albany.edu}

\begin{abstract}
In this paper we will give a simple stopping time proof in the $\Rd$ setting of the matrix weighted Triebel-Lizorkin bounds proved by F. Nazarov/S. Treil and A. Volberg, respectively.  Furthermore, we provide explicit matrix A${}_p$ characteristic dependence and also discuss some interesting open questions. \end{abstract}


\maketitle



\section{Introduction}For any dyadic grid $\D$ in $\mathbb{R}$ and any interval in this grid, let  \begin{equation*} h_I ^1 = |I|^{-\frac{1}{2}} \chi_I (x), \,  \,  \,  \, h_I ^0 (x) = |I|^{-\frac{1}{2}} (\chi_{I_\ell} (x) - \chi_{I_r} (x)).  \end{equation*} Now given any dyadic grid $\D$ in $\mathbb{R},$  a cube $I = I_1 \times \cdots \times I_d$, and any $\varepsilon \in \{0, 1\}^{d}$, let $h_I ^\varepsilon = \Pi_{i = 1}^d h_{I_i} ^\varepsilon$.  It is then easily seen that $\{h_I ^\varepsilon\}_{I \in \D, \  \varepsilon \in \S}$  where $\S = \{0, 1\}^d \backslash \{\vec{1}\}$ is an orthonormal basis for $L^2(\Rd)$.

The classical dyadic Littlewood-Paley estimates state that the Haar basis above is an unconditional basis for $L^p(\Rd)$ when $1 < p < \infty$, and furthermore, \begin{equation} \label{LWPT}\|f\|_{L^p} \approx \left(\inrd \left( \sum_{I \in \D} \sum_{\varepsilon \in \S} \frac{|f_I ^\varepsilon|^2}{|I|} \chi_I (x) \, dx \right)^\frac{p}{2} \, dx \right)^\frac{1}{p}.  \end{equation}  Now let $w$ be an A${}_p$ weight for $1 < p < \infty$, meaning that \begin{equation*} \sup_{\substack{I \subseteq \Rd \\ I \text{is a cube}}} \left(\frac{1}{|I|} \int_I w(x) \, dx \right)\left(\frac{1}{|I|} \int_I w ^{1 - p'} (x) \, dx \right)^{p - 1} < \infty.  \end{equation*}  The $L^p(w)$ version of \eqref{LWPT} first proved in \cite{FJW} states that
\begin{equation} \label{WeightedLWPT}\|f\|_{L^p(w)} \approx \left(\inrd \left( \sum_{I \in \D} \sum_{\varepsilon \in \S} \frac{(m_I w)^\frac{2}{p} | f_I ^\varepsilon|^2}{|I|} \chi_I (x) \, dx \right)^\frac{p}{2} \, dx \right)^\frac{1}{p}  \end{equation} where $m_I w$ is the average of $w$ over $I$.  Note that this says that $L^p(w)$ can be identified as a certain Triebel-Lizorkin space associated to the sequence $\{ (m_I w)^\frac{1}{p}\}_{I \in \D}$ (see \cite{FJW} for definitions).

Now let $n \in \mathbb{N}$ and let $W : \Rd \rightarrow \Mn$ be positive definite a. e.  (where as usual $\Mn$ is the algebra of $n \times n$ matrices with complex scalar entries). Define $L^p(W)$ to be the space of measurable functions $\vec{f} : \Rd \rightarrow \C$ with norm \begin{equation*} \|\vec{f}\|_{L^p(W)} ^p = \inrd |W^\frac{1}{p} (x) \vec{f}(x) |^p \, dx. \end{equation*}  Moreover, assume $W$ is a matrix A${}_p$ weight, meaning that \begin{equation} \label{MatrixApDef} \sup_{\substack{I \subset \R^d \\ I \text{ is a cube}}} \frac{1}{|I|} \int_I \left( \frac{1}{|I|} \int_I \|W^{\frac{1}{p}} (x) W^{- \frac{1}{p}} (t) \|^{p'} \, dt \right)^\frac{p}{p'} \, dx  < \infty \end{equation} where $p'$ is the conjugate exponent of $p$.  Note that matrix A${}_p$ weights arise naturally in various contexts and have drawn a lot of attention recently (see \cite{BPW, CMR, CW, I, IKP, NT, R, TV, V} for example.)

A challenging question (definitively answered in \cite{NT, V} when $d = 1$) is whether \eqref{WeightedLWPT} can be extended to the matrix A${}_p$ setting, and if so what the ``Littlewood-Paley expression" would look like.  To answer this requires some more notation.  It is well known (see \cite{G} for example) that for a matrix weight $W$,  a cube $I$, and any $1 < p < \infty$, there exists positive definite matrices $V_I$ and  $V_I '$ such that  $|I|^{- \frac{1}{p}} \|\chi_I W^\frac{1}{p}  \vec{e}\|_{L^p} \approx |V_I \vec{e}|$ and $|I|^{- \frac{1}{p'}} \|\chi_I W^{-\frac{1}{p}} \vec{e}\|_{L^{p'}} \approx |V_I ' \vec{e}|$ for any $\vec{e} \in \C$, where $\|\cdot \|_{L^p}$ is the
canonical $L^p(\R^d;\C)$ norm and the notation $A \approx B$ as usual means that two quantities $A$ and $B$ are bounded above and below by a constant multiple of each other (which unless otherwise stated will not depend on the weight $W$).  Note that it is easy to see that $\|V_I V_I ' \| \geq 1$ for any cube $I$.  We will say that $W$ is a matrix A${}_p$ weight if the product $V_I V_I'$ has uniformly bounded matrix norm with respect to all cubes $I \subset \R^d$ (note that this condition is easily seen to be equivalent to (\ref{MatrixApDef}), see p. 275 in \cite{R} for example.) Also note that when $p = 2$ we have $V_I = (m_I W)^\frac{1}{2}$ and $V_I ' = (m_I (W^{-1}))^{\frac{1}{2}}$ where $m_I  W$ is the average of $W$ on $I$, so that the matrix A${}_2$ condition takes on a particularly simple form that is similar to the scalar A${}_2$ condition.

Let us emphasize to the reader that while the matricies $V_I$ and $V_I'$ are not averages of $W$, it is nonetheless very useful to think of them as appropriate averages.  With this in mind, it was proved in \cite{NT, V} that the correct matrix weighted version of \eqref{WeightedLWPT} is as follows: if $1 < p < \infty$ and $W$ is a matrix A${}_p$ weight on $\mathbb{R}$, then for any $\vec{f} \in L^p(W)$  we have \begin{equation}\label{LpEmbedding} \|\vec{f}\|_{L^p(W)} \approx \left(\inr \left(\sum_{I \in \D}  \frac{|V_I \vec{f}_I |^2 }{|I|} \chi_I (x) \right)^\frac{p}{2} \, dx \right)^\frac{1}{p} \end{equation}  where $\vec{f}_I$ is the vector of Haar coefficients of the components of $\vec{f}$ with respect to $I$.  Note that \eqref{LpEmbedding} has a number of applications to operators related to singular integral operators (see \cite{IKP, I, NT, V, BPW} for example) and in some sense has the effect of allowing one to ``remove" $W$ and replace it with the $V_I's$.  Furthermore, while is it probably likely that the arguments in \cite{V} can be extended to $\Rd$, \cite{NT} primarily involves Bellman function arguments, which are nontrivial (though not necessarily difficult) to extend to $\Rd$.

The purpose of this paper is to provide a simple stopping time proof of \eqref{LpEmbedding} for $\Rd$ that, unlike \cite{NT, V}, closely resembles known scalar techniques, and in particular is a modification of the stopping time techniques from \cite{KP}  (note that this modification was done in the $p = 2$ setting in \cite{P} for operator valued weights).  Moreover, we also track the A${}_p$ characteristic in \eqref{LpEmbedding}. In particular, we will precisely prove the following, where here $\lceil p \rceil$ is the smallest integer greater or equal to $p$.

\begin{theorem}  \label{MainThm} If $1 < p < \infty$ and $W$ is a matrix A${}_p$ weight then the following two inequalities hold for any $\vec{f} \in L^p(W)$:  \begin{equation*} \|\vec{f}\|_{L^p(W)} \lesssim  \|W\|_{\text{A}_p} ^{ \frac{1}{p} + \frac{{\lceil p \rceil}}{p}} \left(\inrd \left(\sum_{I \in \D} \sum_{\varepsilon \in \S} \frac{|V_I \vec{f}_I ^\varepsilon|^2 }{|I|} \chi_I (x) \right)^\frac{p}{2} \, dx \right)^\frac{1}{p} \end{equation*} and \begin{equation*} \left(\inrd \left(\sum_{I \in \D} \sum_{\varepsilon \in \S} \frac{|V_I \vec{f}_I ^\varepsilon|^2 }{|I|} \chi_I (x) \right)^\frac{p}{2} \, dx \right)^\frac{1}{p} \lesssim \|W\|_{\text{A}_p} ^{ \frac{2}{p} + \frac{{\lceil p' \rceil}}{p}}  \|\vec{f}\|_{L^p(W)} \end{equation*} \end{theorem}

 Note that in some sense ``half" of \eqref{LpEmbedding} was proved in \cite{NT, V} for A${}_{p, \infty}$ weights when $d = 1$ (see either of these for definitions, and see Theorem \ref{ApinfThm} for the precise result).  Furthermore, as we will explain in the last section, our proof, with the help of a Lemma $3.1$ in \cite{V} also proves this.  Also, we will discuss what is known in the literature for the A${}_2$ dependence in \eqref{LpEmbedding} in the case $p = 2$.

Finally, it is hoped that the ``matrixizing" of the ideas in \cite{KP} for $p \neq 2$ will have applications to other challenging yet nonetheless interesting problems involving $L^p(W)$ where cancellation plays a key role.  Furthermore, it would be very interesting to explore whether a Besov space version of Theorem \ref{MainThm} holds (see \cite{R} for results similar to this but outside of the Haar realm.)  In particular, note that Theorem \ref{MainThm} was instrumental in proving (in \cite{IKP}) the equivalency between a dyadic and a continuous characterization of the boundedness of the commutators $[B, R_j]$ on $L^p(W)$  in terms of $B$ (for $B$ a locally integrable $\Mn$ valued function and $R_j$ being any of the Riesz transforms).  Therefore, it is obviously natural to expect that a Besov space version of Theorem \ref{MainThm} would help in proving Schatten $p$ class characterizations of $[B, R_j]$ on $L^2(W)$  in terms of $B$.

\section{Main Lemma}
In this section we will prove a crucial lemma (Lemma \ref{WtdHaarMultThm}).  In fact, the proof of Theorem \ref{MainThm} will follow almost immediately from Lemma \ref{WtdHaarMultThm} and some elementary approximation and duality arguments.  We will first introduce some stopping time definitions from \cite{KP} and prove some preliminary lemmas.  For any dyadic grid $\D$ in $\Rd$ and $I \in \D$, let $\J(I)$ be the collection of dyadic cubes in $\D(I)$ that are maximal with respect to a certain property.  Let $\F(I)$ be the collection of all dyadic cubes in $\D(I)$ that are not contained in any $J \in \J(I)$.  We say the property is admissible if $I \in \F(I)$ for every $I \in \D$.  Given an admissible property, let $\J^0 (I) := I$ and for $j \in \mathbb{N}$ inductively let $\J^j (I)$ be the collection of cubes belonging to $\J(J)$ for some $J \in \J^{j - 1}(I)$.  Similarly define   $\F^j (I)$ be the collection of cubes belonging to $\F(J)$ for some $J \in \J^{j - 1}(I)$. Note that the admissible property clearly gives us that $\D(I) = \cup_{j = 1}^\infty \F^j (I)$. Also clearly the $\F^j(I)$'s are disjoint (with respect to $j$) collections of cubes. We will refer to the stopping time above by simply $\J$.

We will say the stopping time above is decaying if there exists $0 < c < 1$ such that \begin{equation*} \sup_{I \in \D} \frac{1}{|I|} \sum_{J \in \J(I)} |J| \leq c. \end{equation*}    Clearly by iteration this gives us that \begin{equation*} \sup_{I \in \D} \frac{1}{|I|} \sum_{J \in \J^j(I)} |J| \leq c^j  \end{equation*}

Now given a decaying stopping time, fix some large cube $I_0 \in \D$ and for notational simplicity let $\F^j = \F^j (I_0)$ and define $\J^j$ similarly.  Also for some $\vec{f} \in L^p$ let \begin{equation*} \Delta_j \vec{f} = \sum_{I \in \F^j} \sum_{\varepsilon \in \S} \vec{f} _I ^\varepsilon h_I ^\varepsilon \end{equation*}    The proof of the following is exactly the same as the proof of Lemma $7$ in \cite{KP} and therefore will be omitted.  \begin{lemma} \label{KPLem1} Let $\J$ be a decaying stopping time.  For any $1 < p < \infty$ we have that \begin{equation*}  \sum_{j = 0}^\infty \|\Delta_j \vec{f} \|_{L^p} ^p \lesssim \|\vec{f}\|_{L^p} ^p \end{equation*} \end{lemma} Note of course that the implicit constant above depends on the rate of decay of $\J$, but \textit{not} on the cube $I_0$.  As we will see later, the dependence on the decay rate is not important.

The proof of the next lemma is exactly the same as the proof of Lemma $8$ in \cite{KP}.  We will however provide a complete proof as we will need to carefully track the constants involved.

\begin{lemma} Let $\J$ be a decaying stopping time and let $1 < p < \infty$.  Let $T$ be a linear operator on some class of measurable $\C$ valued functions and suppose that $T \vec{f} = \sum_{j = 1}^\infty T_j \vec{f}$ pointwise a.e. for all $\vec{f}$ in this class, where $T_j = T \Delta_j$ and each $T_j \vec{f}$ is measurable.   Assume that there exists $0 < c < 1$ and $C > 0$ where \begin{equation*} \inrd |T_k \vec{f}| ^\frac{p}{2} \, |T_j \vec{f}|^\frac{p}{2} \, dx \leq  C c^{|k - j|} \|\Delta_j \vec{f}\|_{L^p} ^\frac{p}{2} \| \Delta_k \vec{f}\|_{L^p} ^\frac{p}{2}. \end{equation*} Then we have the bound \begin{equation*} \|T \vec{f}\|_{L^p} ^p  \lesssim \frac{C}{(1 - c)^{\lceil p\rceil}} \|\vec{f}\|_{L^p} ^p\end{equation*} for any $\vec{f}$ in this class  (where again the implicit constant above depends on the rate of decay of $\J$ but  \textit{not} on $I_0$.) \end{lemma}

\begin{proof} Let $n = {\lceil p\rceil}$.  Then by Jensen's inequality,  H\"{o}lder's inequality, and the assumption above,  we have \begin{align*}\inrd & |T\vec{f} |^p \, dx  \\ & \leq \inrd \left( \sum_{j = 1}^\infty |T_j \vec{f}| ^\frac{p}{n}  \right)^{n} \, dx  \\ & = \sum_{(j_1, \ldots, j_{n}) \in \N ^{n}} \inrd  |T_{j_1} \vec{f}| ^\frac{p}{n}  \cdots |T_{j_{n}} \vec{f}| ^\frac{p}{n} \, dx \\ & \lesssim \sum_{j_1\leq j_2 \leq \cdots \leq j_{n}} \inrd  |T_{j_1} \vec{f}| ^\frac{p}{n}  \cdots |T_{j_{n}} \vec{f}| ^\frac{p}{n} \, dx  \\ & \leq
\sum_{j_1\leq j_2 \leq \cdots \leq j_{n}} \left(\inrd |T_{j_1} \vec{f}| ^\frac{p}{2} |T_{j_2} \vec{f}| ^\frac{p}{2} \, dx \right) ^\frac{1}{n} \left(\inrd |T_{j_2} \vec{f}| ^\frac{p}{2} |T_{j_3} \vec{f}| ^\frac{p}{2} \, dx  \right) ^\frac{1}{n} \ldots \left(\inrd |T_{j_{n}} \vec{f}| ^\frac{p}{2} |T_{j_1} \vec{f}| ^\frac{p}{2} \, dx \right) ^\frac{1}{n} \\ & \leq C \sum_{j_1\leq j_2 \leq \cdots \leq j_{n}} c^{\frac{j_2 - j_1}{n}} c^{\frac{j_3 - j_2}{n}} \cdots c^{\frac{j_{n} - j_1}{n}} \|\Delta_{j_1} \vec{f} \|_{L^p} ^ \frac{p}{n} \cdots \|\Delta_{j_{n}} \vec{f} \|_{L^p} ^ \frac{p}{n}. \end{align*}  Using H\"{o}lder's inequality one more time and the previous Lemma, we get that \begin{align*}\inrd & |T\vec{f} |^p \, dx  \\ & \leq C \left(\sum_{j_1\leq j_2 \leq \cdots \leq j_{n}} c^{\frac{j_2 - j_1}{n}} c^{\frac{j_3 - j_2}{n}} \cdots c^{\frac{j_{n} - j_1}{n}} \|\Delta_{j_1} \vec{f} \|_{L^p} ^ p \right)^\frac{1}{n} \\ & \times \cdots \times   \left(\sum_{j_1\leq j_2 \leq \cdots \leq j_{n}} c^{\frac{j_2 - j_1}{n}} c^{\frac{j_3 - j_2}{n}} \cdots c^{\frac{j_{n} - j_1}{n}} \|\Delta_{j_{n}} \vec{f} \|_{L^p} ^ p \right)^\frac{1}{n} \\ & \lesssim \frac{C}{(1 - c)^{\lceil p\rceil}} \sum_{j = 1}^\infty \|\Delta_j \vec{f}\|_{L^p} ^p  \\ & \lesssim \frac{C}{(1 - c)^{\lceil p\rceil}}  \|\vec{f}\|_{L^p} ^p.  \end{align*} \end{proof}

Now we will describe our precise stopping time.   For any cube $I \in \D$, let $\J(I)$ be the collection of maximal $J \in \D(I)$ such that \begin{equation}      \|V_J V_I^{-1}\|^p  > \lambda_1 \    \text{   or   }  \  \|V_J^{-1} V_I\|^{p'}  > \lambda_2 \label{STDef} \end{equation} for some $\lambda_1, \lambda_2 > 1$ to be specified later.  Also, clearly $J \in \F(J)$ for any $J \in \D(I)$ so that $\J$ is an admissible stopping time.

We will now show that $\J$ is decaying.  \begin{lemma} \label{DSTLem}
Let $1 < p < \infty$ and let $W$ be a matrix A${}_p$ weight.  For $\lambda_1, \lambda_2 > 1$ large enough, we have that $|\bigcup \J ^j (I)| \leq 2^{-j} |I|$ for every $I \in \D$.  \end{lemma}


\begin{proof} By iteration, it is enough to prove the lemma for $j = 1$.  For $I \in \D$, let $\G(I)$ denote the collection of maximal $J \in \D(I)$ such that the first inequality (but not necessarily the second inequality) in $(\ref{STDef})$ holds.  Then by maximality and elementary linear algebra, we have that \begin{equation*}  \left| \bigcup_{J \in \G(I)} J \right| = \sum_{J \in \G(I)} |J| \lesssim \frac{1}{\lambda_1} \sum_{J \in \G(I)} \int_{J} \|W^\frac{1}{p} (y) V_I ^{-1}\|^p \, dy \leq  \frac{C_1 |I|}{\lambda_1} \end{equation*} for some $C_1 > 0$ only depending on $n$ and $d$.

On the other hand, let For $I \in \D$, let $\widetilde{\G}(I)$ denote the collection of maximal $J \in \D(I)$ such that the second inequality (but not necessarily the first inequality) in $(\ref{STDef})$ holds. Then by the matrix A${}_p$ condition we have  \begin{equation*}  \left| \bigcup_{J \in \widetilde{\G}(I)} J \right|   \leq  \frac{C_2}{\lambda_2}  \sum_{J \in \widetilde{\G}(I)} \int_J \|W^{-\frac{1}{p}} (y) V_I \|^{p'} \, dy  \leq  \frac{C_2 '\|W\|_{A_p} ^\frac{p'}{p}}{\lambda_2} |I| \end{equation*} for some $C_2 '$ only depending on $n$ and $d$.  The proof is now completed by setting $\lambda_1 = 4C_1$ and $\lambda_2 = 4 C_2'  \|W\|_{A_p} ^\frac{p'}{p}$.   \end{proof}

While we will not have a need to discuss matrix A${}_{p, \infty}$ weights in detail in this paper, note that in fact Lemma $3.1$ in \cite{V} immediately gives us that Lemma \ref{DSTLem} holds for matrix A${}_{p, \infty}$ weights (with possibly larger $\lambda_2$ of course.)

Let $\MC{F}$ be the vector space of all $\vec{f} \in \F$.  Now define the constant Haar multiplier $M_{W, p}  $  by \begin{equation*} M_{W, p} \vec{f} := \sum_{I \in \D} \sum_{\varepsilon \in \S}  V_I  {\vec{f}}_I ^\varepsilon h_I ^\varepsilon. \nonumber \end{equation*}  which obviously is defined on $\MC{F}$ and define $M_{W, p} ^{-1}$ on $\MC{F}$ in the obvious way.   We can now state and prove the main result of this section.

\begin{lemma} \label{WtdHaarMultThm} Let $1 < p < \infty$.  If $W$ is a matrix A${}_p$ weight, then $W^\frac{1}{p} M_{W, p} ^{-1} $  satisfies the bounds \begin{equation*} \|W^\frac{1}{p} M_{W, p} ^{-1} \vec{f}\|_{L^p} ^p \lesssim \|W\|_{\text{A}_p} ^{ 1 + {\lceil p \rceil}} \|\vec{f}\|_{L^p} ^p \end{equation*}  for all $\vec{f}$ with finite Haar expansion \end{lemma}

\begin{proof} Fix some $I_0 \in \D$.   Obviously it is enough to prove that the operator $T$ defined on $\F$ by \begin{align} T\vec{f} := \sum_{I \in \D(I_0)} \sum_{\varepsilon \in \S} W^\frac{1}{p} V_I ^{-1} \vec{f}_I ^\varepsilon h_I ^\varepsilon \nonumber \end{align} satisfies the above bound independent of $I_0 \in \D$.  Note that we also clearly have $T \vec{f} = \sum_{j = 1}^\infty T_j \vec{f}$  for $\vec{f} \in \MC{F}$.  Now fix some $\vec{f} \in \MC{F}$.

 For each $I \in \D$, let \begin{align}{M_I} \vec{f} := \sum_{J \in \F(I)} \sum_{\varepsilon \in \S}  V_J ^{-1} \vec{f}_J ^\varepsilon h_J ^\varepsilon \nonumber \end{align} so that \begin{align} T_j \vec{f} = \sum_{I \in \J^{j - 1}} W^\frac{1}{p} {M_I} \vec{f}. \nonumber \end{align}   Since $V_I M_I$ is a constant Haar multiplier and since $\|V_I V_J ^{-1}\| ^p \lesssim  \|W\|_{\text{A}_p}$ if $J \in \F(I)$, we immediately have that

 \begin{equation} \|V_I {M_I} \vec{f}\|_{L^p} ^p  \lesssim \|W\|_{A_p} \|\vec{f}\|_{L^p}^p. \nonumber \end{equation}


Now we will show that each $T_j$ is bounded.  To that end, we have that \begin{align}  \int_{\Rd} |T_j f|^p \, dx & = \int_{\bigcup_{\J^{j - 1}}  \backslash \bigcup_{\J^j}}  |T_j f|^p \, dx +  \int_{\bigcup_{\J^j}} |T_j f|^p \, dx \nonumber \\ & := (A) + (B). \nonumber \end{align} Since $\| W^\frac{1}{p} (x) V_J ^{-1}\| ^p \lesssim 1$ on $J \backslash \bigcup \J(J)$, we can estimate $(A)$ first as follows: \begin{align} (A) & = \sum_{J \in \J^{j - 1}} \int_{J \backslash \bigcup \J(J)} |T_j \vec{f}|^p \, dx \nonumber \\ & = \sum_{J \in \J^{j - 1}} \int_{J \backslash \bigcup \J(J)} |W^\frac{1}{p} (x) M_J \vec{f} (x)|^p \, dx \nonumber \\ & \leq \sum_{J \in \J^{j - 1}} \int_{J \backslash \bigcup \J(J)} \| W^\frac{1}{p} (x) V_J ^{-1}\| ^p |V_J  M_J  \vec{f} (x)|^p \, dx  \nonumber \\ & \lesssim  \sum_{J \in \J^{j - 1}} \int_{J } |V_J  M_J  \vec{f} |^p\, dx \nonumber \\  &   \lesssim \|W\|_{A_p} \|\vec{f}\|_{L^p} ^p. \nonumber \end{align}  As for $(B)$, note that $M_J \vec{f}$ is constant on $I \in \J(J)$, and so we will refer to this constant by $M_J \vec{f}(I)$.  We then estimate $(B)$ as follows:
\begin{align} (B) & = \int_{\bigcup \J^j} |T_j \vec{f}|^p \, dx \nonumber \\ & \leq \sum_{J \in \J^{j - 1}} \sum_{I \in \J(J)} \int_I |W^\frac{1}{p} (x) M_J \vec{f}|^p \, dx \nonumber \\ &  \leq \sum_{J \in \J^{j - 1}} \sum_{I \in \J(J)} |I| | V_J  M_J \vec{f}(I) | \left (\frac{1}{|I|} \int_I \|W^\frac{1}{p} (x) V_J ^{-1} \|^p \, dx \right)  \nonumber \\ & \lesssim \sum_{J \in \J^{j - 1}} \sum_{I \in \J(J)} |I| | V_J  M_J \vec{f}(I) | \nonumber \\ & = \sum_{J \in \J^{j - 1}} \sum_{I \in \J(J)}  \int_I | V_J  M_J \vec{f}(x) |^p  \, dx \nonumber \\ &  \lesssim \|W\|_{A_p}  \|\vec{f}\|_{L^p} ^p.  \label{WtdHaarMultThm3} \end{align}  However, clearly $T_j \vec{f} = T_j \Delta_j \vec{f}$ so that \begin{equation*} \int_{\Rd} |T_j f|^p \, dx \lesssim \|W\|_{\text{A}_p} \|\Delta_j \vec{f}\|_{L^p} ^p  \end{equation*}

To finish the proof, we claim that there exists $0 < c <  1$ such that \begin{align} \int_{\bigcup \J^{k - 1} } |T_j \vec{f}|^p \, dx  \lesssim c^{k - j} \|\vec{f}_j\|_{L^p} ^p \nonumber \end{align} whenever $k > j$.  If we define $M_j \vec{f}$ as \begin{align} M_j \vec{f} : = \sum_{I \in \J^{j - 1}} M_I \vec{f}, \nonumber \end{align} then $M_j \vec{f}$ is constant on $J \in \J^j$. Thus, we have that \begin{align} \int_{\bigcup \J^{k - 1} } |T_j \vec{f}|^p \, dx & = \sum_{J \in \J^j} \sum_{I \in \J^{k - j - 1} (J)} \int_I |W^\frac{1}{p} (x) M_j\vec{f} (J)|^p \, dx \nonumber \\ & \leq \sum_{J \in \J^j} \sum_{I \in \J^{k - j - 1} (J)} |J||V_J M_j \vec{f} (J)  | ^p \left(\frac{1}{|J|} \int_I \| W^\frac{1}{p} (x)  V_J ^{-1} \|^p \,dx \right). \nonumber \end{align} However, \begin{align}   |J| |V_J M_j \vec{f} (J)  | ^p  & \lesssim \int_J | W^\frac{1}{p} (x) M_j \vec{f} (J)|^p \, dx \nonumber \\ & = \int_J |T_j\vec{f}(x)|^p\,dx.  \label{WtdHaarMultThm4} \end{align} On the other hand, it is not hard to show that $|W ^\frac{1}{p} (x) \vec{e}|^p$ is a scalar A${}_p$ weight for any $\vec{e} \in \C$ (see \cite{G}), which by the classical reverse H\"{o}lder inequality means that we can pick some $q > p$ and use H\"{o}lder's inequality in conjunction with Lemma $\ref{DSTLem}$ to get \begin{align} \frac{1}{|J|} \sum_{I \in \J^{k - j - 1} (J)} \int_I & \| W^\frac{1}{p} (x)  V_J ^{-1} \|^p \,dx  \nonumber \\ & = \frac{1}{|J|}  \int_{\bigcup \J^{k - j - 1}(J)} \|W^\frac{1}{p} (x)  V_J ^{-1} \|^p \,dx \nonumber \\ & \leq \frac{1}{|J|} \left(\int_{\bigcup \J^{k - j - 1}(J)} \| W^\frac{1}{p} (x) V_J ^{-1} \|^q \, dx \right)^\frac{p}{q} (2 ^{-(k - j - 1)} |J|) ^{1 - \frac{p}{q}} \nonumber \\ & \lesssim 2 ^{-(k - j - 1)(1 - \frac{p}{q})} \left(\frac{1}{|J|} \int_J \|W^\frac{1}{p} (x) V_J ^{-1} \|^q \, dx \right)^\frac{p}{q} \nonumber \\ & \lesssim  2 ^{-(k - j - 1)(1 - \frac{p}{q})}. \label{WtdHaarMultThm5} \end{align}  Combining $(\ref{WtdHaarMultThm4})$ with $(\ref{WtdHaarMultThm5})$, we get that \begin{align} \int_{\bigcup \J^{k - 1} } |T_j \vec{f}|^p \, dx & \lesssim 2 ^{-(k - j - 1)(1 - \frac{p}{q})}   \int_{\bigcup_{\J^j}} |T_j \vec{f}  |^p \, dx \nonumber \\ & \lesssim \|W\|_{A_p}   2^{-(k - j - 1)(1 - \frac{p}{q})} \|\Delta_j \vec{f}\|_{L^p} ^p. \nonumber \end{align}   Note however, that the scalar characteristic of each $\|W\|  | \vec{e}|$ is less than or equal to $\|W\|_{A_p} |\vec{e}|$ which means that classically we can choose $q$ such that \begin{equation*} q = p + \frac{1}{C\|W\|_{\text{A}_p} } \end{equation*} for some $C > 1$ independent of $W$, so that \begin{equation*} 2^{- (1 - \frac{p}{q})} = 2^{- \frac{1}{q C \|W\|_{\text{A}_p}}} = 2^{- \frac{1}{p C  \|W\|_{\text{A}_p} + 1 }} \end{equation*}  which means that \begin{equation*} \int_{\bigcup \J^{k - 1} } |T_j \vec{f}|^p \, dx \lesssim \|W\|_{\text{A}_p} 2^{- \frac{k - j - 1}{p C  \|W\|_{\text{A}_p} + 1 }}\|\Delta_j \vec{f}\|_{L^p} ^p \end{equation*} since again $T_j \vec{f} = T_j \Delta_j \vec{f}$.

  Finally,  this gives us that \begin{align*} \inrd |T_k \vec{f}| ^\frac{p}{2} \, |T_j \vec{f}|^\frac{p}{2} \, dx & \leq \left(\int_{\bigcup \J^{k - 1} } |T_j \vec{f}| ^p \, dx \right)^\frac{1}{2} \|T_k \vec{f}\|_{L^p } ^\frac{p}{2}  \\ & \lesssim  (\|W\|_{\text{A}_p} 2^{- \frac{k - j - 1}{p C  \|W\|_{\text{A}_p} + 1 }}\|\Delta_k \vec{f}\|_{L^p} ^p)^\frac{1}{2}  (\|W\|_{\text{A}_p} \|\Delta_j \vec{f}\|_{L^p} ^p)^\frac12 \end{align*} and combining this with the previous two lemmas gives us that

  \begin{align*} \|W^\frac{1}{p} M_{W, p} ^{-1} \vec{f} \|_{L^p} ^p & \lesssim \|W\|_{\text{A}_p} \left(\frac{1}{1 - 2^{-\frac{1}{2 pC \|W\|_{\text{A}_p}}}}\right)^{\lceil p \rceil} \|\vec{f}\|_{L^p} ^p   \\ & \lesssim     \|\vec{f}\|_{L^p} ^p \|W\|_{\text{A}_p} ^{ 1 + {\lceil p \rceil}} \end{align*}

\end{proof}

Note that the proof of Lemma \ref{WtdHaarMultThm} only requires Lemma \ref{DSTLem} and the fact that $|W^\frac{1}{p} (x) \vec{e} |^p$ satisfies a reverse H\"{o}lder inequality for each $\vec{e}$. In particular, our proof (as do proofs in \cite{NT, V}) holds for matrix A${}_{p, \infty}$ weights.

   \section{Proof of Theorem \ref{MainThm} and comments}

   We will first prove Theorem \ref{MainThm} for $C_c^\infty(\Rd)$ functions, which largely involves Lemma \ref{WtdHaarMultThm} in conjunction with some easy duality arguments.   Theorem \ref{MainThm} in general will then follow from the fact that $C_c^\infty(\Rd)$ is dense in $L^p(W)$ (see Proposition $3.7$ in \cite{CMR}, which interestingly holds for not necessarily matrix A${}_p$ weights) and some elementary approximation arguments. For the rest of this section, pick strictly increasing finite subsets $\{F_k\}_{k = 1}^\infty $ of $\D$ whose union is all of $\D$. \\

\noindent \textit{Proof of  Theorem \ref{MainThm}}.  We first prove the first inequality in Theorem \ref{MainThm} for $C_c^\infty(\Rd)$ functions.  Let $\vec{f} \in C_c^\infty(\Rd)$ and let $\vec{f}_k$ be defined by \begin{equation*} \vec{f}_k = \sum_{I \in F_k} \sum_{\varepsilon \in \S} \vec{f}_I ^\varepsilon h_I ^\varepsilon \end{equation*} so that obviously $M_{W, p} \vec{f}_k$ is well defined. Also note that obviously we have $\F \subseteq L^p(W).$  Then from Lemma \ref{WtdHaarMultThm} we have that \begin{align*} \|\vec{f}_k \|_{L^p (W)} ^p & = \|W^\frac{1}{p} \vec{f}_k\|_{L^p} ^p \\ & = \|(W^\frac{1}{p} M_{W, p} ^{-1} ) M_{W, p}  \vec{f}_k \|_{L^p} ^p \\ & \lesssim \|W\|_{\text{A}_p} ^{ 1 + {\lceil p \rceil}}  \|M_{W, p}  \vec{f}_k\| _{L^p} ^p \\ & \lesssim  \|W\|_{\text{A}_p} ^{ 1 + {\lceil p \rceil}} \inrd \left(\sum_{I \in \D} \sum_{\varepsilon \in \S} \frac{|V_I \vec{f}_I ^\varepsilon|^2 }{|I|} \chi_I (x) \right)^\frac{p}{2} \, dx \end{align*} by standard dyadic Littlewood-Paley theory. However, since $\vec{f} \in C_c^\infty(\Rd)$ we have that $\vec{f}_k \rightarrow \vec{f}$ pointwise, so finally by Fatou's lemma

   \begin{align*} \|\vec{f}\|_{L^p(W)} ^p & \leq \liminf_{k \rightarrow \infty} \|\vec{f}_k \|_{L^p(W)} ^p \\ & \lesssim  \|W\|_{\text{A}_p} ^{ 1 + {\lceil p \rceil}} \inrd \left(\sum_{I \in \D} \sum_{\varepsilon \in \S} \frac{|V_I \vec{f}_I ^\varepsilon|^2 }{|I|} \chi_I (x) \right)^\frac{p}{2} \, dx. \end{align*}

 Now we prove the second inequality in Theorem \ref{MainThm} for $C_c^\infty(\Rd)$ functions.  If $\vec{f} \in C_c^\infty(\Rd) $ then  let $\vec{\Psi}_k$ be defined by \begin{equation*} \vec{\Psi}_k = \sum_{I \in F_k} \sum_{\varepsilon \in \S} V_I ^{-1}  \vec{f}_I ^\varepsilon h_I ^\varepsilon. \end{equation*} Also for $\vec{g} \in \F$, define $\vec{g}_k$ as before and let $F_{\vec{g}}$ be the class of $I \in \D$ such that the Haar coefficient of $\vec{g}$ with respect to $I$ is not zero. Then clearly \begin{align*} | \langle \vec{\Psi}_k , \vec{g}\rangle_{L^2}| & = \left| \sum_{I \in F_k \cap F_{\vec{g}}} \sum_{\varepsilon \in \S} \langle V_I ^{-1} \vec{f} _I ^\varepsilon , \vec{g}_{I} ^\varepsilon \rangle_{\C} \right| \\ & = \left| \sum_{I \in \D} \sum_{\varepsilon \in \S} \langle   \vec{f} _I ^\varepsilon , V_I ^{-1} (\vec{g}_k) _{I} ^\varepsilon \rangle_{\C} \right| \\ & \leq \inrd | \langle \vec{f} ,  M_{W, p} ^{-1} \vec{g}_k \rangle_{\C} | \, dx  \\ & = \inrd | \langle W^{-\frac{1}{p}} \vec{f} , W^{\frac{1}{p}}  M_{W, p} ^{-1}\vec{g}_k \rangle_{\C} | \, dx  \\ & \lesssim \|W\|_{\text{A}_p} ^{ \frac{1}{p} + \frac{\lceil p \rceil}{p}} \|W^{-\frac{1}{p}} \vec{f}\|_{L^{p'}} \|\vec{g}\|_{L^p}  \end{align*} By a trivial approximation argument we therefore have \begin{equation*} \sup_k \|\vec{\Psi} _k \|_{L^{p'}}  \lesssim \|W\|_{\text{A}_p} ^{ \frac{1}{p} + \frac{ \lceil p \rceil}{p}} \|W^{-\frac{1}{p}} \vec{f}\|_{L^{p'}}  \end{equation*} which by dyadic Littlewood-Paley theory and the monotone convergence theorem says

\begin{align*}  & \left(\inrd \left(\sum_{I \in \D} \sum_{\varepsilon \in \S} \frac{|V_I ^{-1} \vec{f} _I ^\varepsilon|^2 }{|I|} \chi_I (x) \right)^\frac{p'}{2} \, dx\right)^{\frac{1}{p'}} \\ & = \lim_{k \rightarrow \infty} \left(\inrd \left(\sum_{I \in F_k} \sum_{\varepsilon \in \S} \frac{|V_I ^{-1} \vec{f}_I ^\varepsilon|^2 }{|I|} \chi_I (x) \right)^\frac{p'}{2} \right)^{\frac{1}{p'}}\\ & \lesssim \sup_k  \|\vec{\Psi} _k \|_{L^{p'}}  \\ & \lesssim \|W\|_{\text{A}_p} ^{ \frac{1}{p} + \frac{{\lceil p \rceil}}{p}} \|W^{-\frac{1}{p}} \vec{f}\|_{L^{p'}} \end{align*} which means that  \begin{align*} & \left(\inrd \left(\sum_{I \in \D} \sum_{\varepsilon \in \S} \frac{|V_I '   \vec{f}_I ^\varepsilon |^2 }{|I|} \chi_I (x) \right)^{\frac{p'}{2}} \, dx\right)^{\frac{1}{p'}} \\ & \leq  \|W\|_{\text{A}_p} ^\frac{1}{p} \left(\inrd \left(\sum_{I \in \D} \sum_{\varepsilon \in \S} \frac{|V_I ^{-1}  \vec{f}_I ^\varepsilon|^2 }{|I|} \chi_I (x) \right)^\frac{p'}{2} \, dx \right)^{\frac{1}{p'}} \\ &  \lesssim \|W\|_{\text{A}_p} ^{ \frac{2}{p} + \frac{{\lceil p \rceil}}{p}} \|W^{-\frac{1}{p}}\vec{f}\|_{L^{p'}}  \end{align*} by the matrix A${}_p$ condition.  However,  $W$ is a matrix A${}_p$ weight if and only if $W^{1 - p'}$ is a matrix A${}_{p'}$ weight with $\|W^{1 - p'}\|_{\text{A}_{p'}} = \|W\|_{\text{A}_p} ^{\frac{1}{p - 1}}$.  Furthermore, one can easily check that  $V_I ' (W^{1 - p'}, p') =  V_I(W, p)$ which then means that \begin{align*} & \left(\inrd \left(\sum_{I \in \D} \sum_{\varepsilon \in \S} \frac{|V_I    \vec{f}_I ^\varepsilon |^2 }{|I|} \chi_I (x) \right)^{\frac{p}{2}} \, dx\right)^{\frac{1}{p}} \\ & \lesssim  \|W^{1 - p'}\|_{\text{A}_{p'}} ^{ \frac{2}{p'} + \frac{{\lceil p' \rceil}}{p' }} \|\vec{f}\|_{L^p(W)} \\ & = \|W\|_{\text{A}_{p}} ^{ \frac{2}{(p - 1)p'} + \frac{{\lceil p' \rceil}}{(p - 1)p' }} \|\vec{f}\|_{L^p(W)} \\ & = \|W\|_{\text{A}_{p}} ^{ \frac{2}{p} + \frac{{\lceil p' \rceil}}{p }} \|\vec{f}\|_{L^p(W)} \end{align*}which completes the proof for $C_c^\infty(\Rd)$ functions.

Finally we will complete the proof of Theorem \ref{MainThm} for $\vec{f} \in L^p(W)$.  Pick a sequence $\{\vec{f}_k\}_{k = 1} ^\infty \subseteq C_c^\infty(\Rd)$ where $\vec{f}_k \rightarrow f$ in $L^p(W)$.  Note that since $W^{1 - p'}$ is locally integrable we clearly have that $(\vec{f}_k)_I ^\varepsilon \rightarrow \vec{f}_I ^\varepsilon$ for any $I \in \D$ and $\varepsilon \in \S$.  Then by Theorem \ref{MainThm} for $ C_c ^\infty(\Rd)$ functions, Fatou's lemma, and the two facts stated above, we have \begin{align*}   \|W\|_{\text{A}_{p}} ^{ {2} + {\lceil p' \rceil}}  \|\vec{f}\|_{L^p(W)} ^p  & =  \|W\|_{\text{A}_{p}} ^{ {2} + {\lceil p' \rceil}} \liminf_{k \rightarrow \infty} \|\vec{f}_k\|_{L^p(W)} ^p  \\ & \gtrsim  \liminf_{k \rightarrow \infty} \inrd \left(\sum_{I \in \D} \sum_{\varepsilon \in \S} \frac{|V_I (\vec{f}_k)_I ^\varepsilon| ^2 }{|I|} \chi_I (x) \right)^\frac{p}{2} \, dx \\ & \geq   \inrd \left(\sum_{I \in \D} \sum_{\varepsilon \in \S} \frac{|V_I \vec{f}_I ^\varepsilon| ^2 }{|I|} \chi_I (x) \right)^\frac{p}{2} \, dx  \end{align*} which proves half of Theorem \ref{MainThm} for $\vec{f} \in L^p(W)$, and additionally proves that \begin{equation*} \lim_{k \rightarrow \infty} \inrd \left(\sum_{I \in \D} \sum_{\varepsilon \in \S} \frac{|V_I (\vec{f}_k - \vec{f})_I ^\varepsilon| ^2 }{|I|} \chi_I (x) \right)^\frac{p}{2} \, dx   = 0. \end{equation*} The other half now follows immediately since \begin{align*} \|\vec{f}\|_{L^p(W)} ^p  &  = \lim_{k \rightarrow \infty} \|\vec{f}_k \|_{L^p(W)} ^p \\ &  \lesssim \|W\|_{\text{A}_p} ^{ 1 + {\lceil p \rceil}}   \lim_{k \rightarrow \infty} \inrd \left(\sum_{I \in \D} \sum_{\varepsilon \in \S} \frac{|V_I (\vec{f}_k)_I ^\varepsilon| ^2 }{|I|} \chi_I (x) \right)^\frac{p}{2} \, dx  \\ & = \|W\|_{\text{A}_p} ^{ 1 + {\lceil p \rceil}}  \inrd \left(\sum_{I \in \D} \sum_{\varepsilon \in \S} \frac{|V_I \vec{f}_I ^\varepsilon| ^2 }{|I|} \chi_I (x) \right)^\frac{p}{2} \, dx \end{align*} \hfill $\Box$

Let us finish this paper with two comments.  As was earlier commented, Lemma \ref{WtdHaarMultThm} holds for matrix A${}_{p, \infty}$ weights.  Because of this, the attentive reader will notice that we have actually proved the following, which was also proved in \cite{NT, V} when $d = 1$.

\begin{theorem} \label{ApinfThm} If $W$ is a matrix A${}_{p, \infty}$ weight and $1 < p < \infty$ then \begin{equation*} \left(\inrd \left(\sum_{I \in \D} \sum_{\varepsilon \in \S} \frac{|V_I ^{-1}  \vec{f}_I ^\varepsilon|^2 }{|I|} \chi_I (x) \right)^\frac{p'}{2} \, dx\right)^{\frac{1}{p'}} \lesssim \|\vec{f}\|_{L^{p'} (W^{1 - p'})} \end{equation*}  \end{theorem}

Finally, note that when $p = 2$ and $W$ is a matrix A${}_2$ weight, our results considerably simplify to the inequalities  \begin{equation*} \|\vec{f}\|_{L^2(W)} \lesssim  \|W\|_{\text{A}_2} ^\frac{3}{2} \left(\sum_{I \in \D} \sum_{\varepsilon \in \S} |(m_I W)^\frac12  \vec{f}_I ^\varepsilon|^2 \right)^\frac{1}{2} \end{equation*} and \begin{equation*} \left(\sum_{I \in \D} \sum_{\varepsilon \in \S} |(m_I W)^\frac12  \vec{f}_I ^\varepsilon|^2 \right)^\frac{1}{2} \lesssim \|W\|_{\text{A}_2} ^{ 2} \|\vec{f}\|_{L^2(W)}.  \end{equation*} Note that these ``square function bounds" were obtained first in \cite{P} (where the A${}_2$ characteristic dependence was not tracked)  also using stopping techniques that are similar to the ones in \cite{KP}. Moreover, these square function bounds were proved in \cite{BPW, CW} using vastly different techniques but with the better bounds of $\|W\|_{\text{A}_2} ^\frac12  (\log \|W\|_{\text{A}_2} )^\frac12 $ and $\|W\|_{\text{A}_2}  (\log \|W\|_{\text{A}_2} )^\frac12 $, respectively.  Finally, it is well known and was proved in \cite{PP, HTV} that in the scalar A${}_2$ setting the bounds of $\|W\|_{\text{A}_2}  ^\frac12 $ and $\|W\|_{\text{A}_2} $ are sharp.  It appears to be a very interesting but challenging problem as to whether these bounds hold in the matrix weighted setting.

 \end{document}